\documentclass[11pt]{amsart}

\usepackage{amsmath,amssymb,amsthm,comment,enumitem,mathtools,tikz,ytableau}
\usepackage{hyperref}
\usepackage[margin=1.5in]{geometry}
\usepackage[all]{hypcap}

\ytableausetup{centertableaux}

\newtheorem{theorem}{Theorem}

\newtheorem{corollary}[theorem]{Corollary}
\newtheorem{lemma}[theorem]{Lemma}
\newtheorem{proposition}[theorem]{Proposition}
\newtheorem{remark}[theorem]{Remark}

\theoremstyle{definition}
\newtheorem{definition}[theorem]{Definition}
\newtheorem{notation}[theorem]{Notation}

\newcommand{\qmat}{\lbrace q\rbrace_{\alpha_2}}
\newcommand{\lb}{\lbrace}
\newcommand{\rb}{\rbrace}
\newcommand{\al}{\alpha}
\newcommand{\nts}{\negthickspace}
\newcommand{\lf}{\lfloor}
\newcommand{\rf}{\rfloor}

\title{Solutions to $\sum_{i=1}^n 1/x_i=1$ in integers $p^a\,q^b$ with $p$ and $q$ two set primes}

\author{Claire Levaillant}

\begin{document}

\begin{abstract} We present an algorithm for computing all the solutions in not necessarily distinct integers to the decomposition of the unit into a sum of unit fractions with denominators $p^a.q^b$ where $p$ and $q$ are two distinct primes, each appearing at least once in the solution. 
\end{abstract}

\maketitle

\noindent\textbf{Acknowledgements.} This paper grew out of a joint work with Joel Louwsma. Because his algorithm and programming differ from those of the present paper, out of a common agreement, a future work will appear independently, which also betters some of the results contained in this paper. The author is indebted to him for many fruitful conversations and for allowing her to state and use his yet unpublished Proposition~\ref{prop:construction} for the sake of the discussion contained in the current paper.  A better version of the latter result is scheduled to appear later. The author would like to acknowledge the mathematics department of the University of Southern California, where some parts of her CAML program were changed and bettered. She is grateful to Yassine El Maazouz for generously printing reference \cite{CUR} during her short visit at the mathematics department of Caltech. 

\section{Introduction and preliminary results}
We investigate the solutions to the Diophantine equation $\sum_{i=1}^n 1/x_i=1$ in integers of the form $p^a q^b$, where $p$ and $q$ are two distinct primes, $a\geq 0$, $b\geq 0$, the primes $p$ and $q$ are the only divisors involved in the denominators, and they both appear as factor in the denominators at least once. In particular, the two primes may vary from a solution to another solution, but may not vary from a unit fraction to another unit fraction.

Studying the solutions to the Diophantine equation $\sum_{i=1}^n 1/x_i=1$ in distinct integers and more generally decomposing any positive rational number into a sum of distinct unit fractions, a so-called ``Egyptian fraction", goes back in history as far as an ancient Egyptian text, namely the Rhind Mathematical Papyrus, dating back to around $1650$ B.C. \cite{ROB}.
In not necessarily distinct integers, one interest in the solutions to this Diophantine equation relies in the fact that finding such solutions to the Diophantine equation is equivalent to finding $n$ positive integers $k_i$'s with $1\leq i\leq n$, with no common factor, such that:
\[k_j\bigg|\sum_{i=1}^n k_i,\qquad\forall 1\leq j\leq n;\]
see for instance Theorem $4.4$ of \cite{KER} proven in a more general setting. This problem is itself equivalent to finding an arithmetical structure on the complete graph $K_n$. Arithmetical structures were originally introduced by Lorenzini in \cite{LOR} in order to study intersections of degenerating curves in algebraic geometry. From the matricial formulation of Lorenzini of an arithmetical structure, appears a finite abelian group called ``critical group", which is a generalization of the sandpile group; see \cite{BIG}. The latter group itself relates to the chip-firing game on a graph; see e.g. \cite{GLA} and \cite{BLS}.

The number of arithmetical structures on $K_n$ appears in the online encyclopedia of integer sequences \cite{OEI} for $n\leq 8$. The case $n=9$ is specifically addressed in \cite{LEV}, working only in distinct integers and adopting the viewpoint of the Diophantine equation.The study, which also generalizes to any decomposition length, is restricted to solutions involving exactly two prime factors in the denominators, one of which is $2$. The other (odd) prime may not vary from a unit fraction to another unit fraction, but may vary from a solution to another solution. The main restriction however is on the highest exponent of $2$ which should not exceed $2$. An independent study in \cite{LOU}, which uses a very different method,  generalizes the latter work by additionally allowing for non-distinct integers and the odd prime factor of the denominators gets replaced with any positive integer that is not a power of $2$. However, the same restriction on the highest exponent of $2$ is crucially maintained. This restriction gets waived in the current paper, which offers a much wider framework than \cite{LEV}, by considering solutions in not necessarily distinct integers involving any two distinct primes, without any restriction on their respective highest exponents. When the integers are odd, the minimal length of a representation of the unit as a sum of distinct unit fractions is $9$. In \cite{BUR}, the author had shown that there are exactly five solutions in distinct odd integers when $n=9$, all of which involved at least three odd primes. The current study, applied with $p=2$ and distinct integers, thus allows in particular to find all the solutions in $9$ distinct integers involving only two primes.

Our approach is different from the almost century old approach of Erd\"os and Graham in that they set the number of primes involved in the denominators, but they allow these primes to vary from a unit fraction to the other, while we don't. Their work gathered in \cite{BUT} led to showing that each natural number has a representation in Egyptian fraction with each denominator having three distinct prime divisors. Likewise in \cite{BU2}, Burshtein discusses the case of two distinct prime factors while allowing these primes to vary from a unit fraction to the other. Other authors have been following our approach by setting both the number of primes involved and setting these primes, while their main interest relied in counting the solutions and finding some bounds on these counts. Their work appears in \cite{CHE}, where they only deal with Egyptian fraction decompositions of the unit. Contrary to us, they do not exhibit the solutions. Our study provides an algorithm for exhibiting all the solutions to the Diophantine equation, under some given restrictions. After these restrictions are set, we do not limit ourselves to bounding the number of solutions, nor do we limit ourselves to exhibiting only a few of them.

We will record a solution as an array of integers $\{k_{a,b}\}_{a,b\geq0}$, where  $k_{a,b}$ records the number of times the term $1/p^aq^b$ appears in the solution. We represent this array graphically, with $a$ increasing along the columns from left to right, and $b$ increasing along the rows, from bottom to top. 
For example, letting $n=7$ and setting $p=2$ and $q=3$, the solution
\[\frac{1}{2}+\frac{1}{2^2}+\frac{1}{2^3\cdot3}+\frac{1}{2^3\cdot3}+\frac{1}{2\cdot3^2}+\frac{1}{2\cdot3^2}+\frac{1}{2\cdot3^2}=1,\] gets represented as:
\[
\begin{ytableau}
0 & 3 & 0 & 0 \\
0 & 0 & 0 & 2 \\
0 & 1 & 1 & 0\\
\end{ytableau}
\]
\[\begin{array}{l}\end{array}\]
We will describe an algorithm which, provided the primes $p$ and $q$ involved as factors in the denominators, the number $n$ of unknowns of the Diophantine equation, and the width $\al_p+1$ of the table in entry, returns all the solutions to the equation with highest $p$-valuation at most $\al_p$.
From there, 
fix the width $\al_p+1$ of the table and $n$ the number of unknowns. As part of our work, we are able to bound the primes $q$ occurring in solutions with highest $p$-valuation at most $\al_p$. For a given prime $p$ and a given width, we thus obtain all the solutions to the equation in $n$ integer variables, involving all possible other primes $q$. The present work can thus be viewed as a generalization of \cite{LEV} in the case when $p$ is any prime instead of $2$, $\al_p>2$, and the integers are not necessarily distinct. Another versant of our work is the following. We can set the primes $p$ and $q$ instead of setting one prime and letting the other prime vary from a solution to another solution. As a matter of fact, the largest denominator in the decomposition of the unit into a sum of $n$ unit fractions is bounded strictly by the Sylvester sequence $S_n$. The Sylvester sequence is defined by
\[S_1=2,\qquad\forall i\geq 1,\,S_{i+1}=1+\prod_{k=1}^i S_k.\]
This fact was shown by Curtiss in \cite{CUR}, upon a question raised by Kellogg a year earlier in \cite{KEL}. Thus, denoting by $\al_p\geq 1$ the  highest $p$-valuation of a solution with $n$ integer variables, we have
\[p^{\al_p}<S_n,\]
forcing in turn
\[\al_p<log_p(S_n).\]
By running the algorithm on all the possible table widths, we thus obtain all the solutions of length $n$ in $p^aq^b$, with $a\geq 0$, $b\geq 0$, and $p$ and $q$ each appearing at least once.

The paper is structured as follows. The first part of the paper focuses on the repartition of the primes $q$ occurring in solutions when the other prime $p$, the highest $p$-valuation and the decomposition length are all fixed. As part of our work, we establish a sufficient condition for the existence of a solution in the special case when $p=2$ and we also prove a necessary condition for the existence of a solution in the general case, for sufficiently large $q$. The second half of the paper is devoted to a full description of the algorithm. \\

\section{Repartition of the primes}

Our discussion will be based on the following fundamental lemma.
\begin{lemma}[Admissibility of a top row]\label{prop:admissibility}
If $k_1\,\,k_2\,\dotsm\, k_{\alpha_p+1}$ is the top row of a solution in $p^aq^b$ with $b\geq 0$ and $0\leq a\leq \al_p$, then $q$ divides $\sum_{i=1}^{\alpha_p+1}k_ip^{\alpha_p+1-i}$.
\end{lemma}

\begin{proof} Denote by $\al_q$ the highest $q$-valuation. By reducing to the same denominator we obtain:
\[p^{\al_p}k_1+p^{\al_p-1}k_2+\dots + k_{\al_p+1}+qm=p^{\al_p}q^{\al_q},\]
for an appropriate integer $m$. The result follows.
\end{proof}

In light of Lemma~\ref{prop:admissibility}, we call a top row \emph{admissible} if $q$ divides $\sum_{i=1}^{\alpha_p+1}k_ip^{\alpha_p+1-i}$.

\begin{definition}[Push value] When referring to the top row of a solution $k_1\;\dots\,k_{\al_p+1}$, we call push value the integer resulting from the quotient $\sum_{i=1}^{\alpha_p+1}k_ip^{\alpha_p+1-i}/q$.
\end{definition}
Notice that
\[\sum_{i=1}^{\al_p+1}k_ip^{\al_p-i+1}\leq p^{\al_p}\sum_{i=1}^{\al_p+1}k_i\leq p^{\al_p}n.\]
Then, by Lemma~\ref{prop:admissibility} we have: \[q <p^{\al_p}n.\]
We will improve this bound in forthcoming Corollary~\ref{cor:boundimprovement}.

The case $p=2$ is special in that if we
fix $\al_2$ and $q$, we can show that there is always a solution to the Diophantine equation for sufficiently large $n$. First, we introduce a few more notations.
\begin{notation}
For a prime $p$ and an integer $m$, we will denote by $\lb m\rb_{\al_p}$ the residue of $m$ modulo $p^{\al_p}$.
\end{notation}
\begin{notation}
Denote any integer $\ell=a_0+pa_1+\dots+a_kp^k$ with $a_i\in\lb 0,1,\dots, p-1\rb$ for each $i$ with $0\leq i\leq k$ written (uniquely) in base $p$ by
$\overline{a_ka_{k-1}\dotsm a_0}^{(p)}$. \\
Let $N(\ell)\coloneqq\sum_{i=0}^k{a_i}$. In the case when $p=2$, $N(\ell)$ counts the number of $1$'s in the binary expansion of $\ell$ and will be specially denoted by $N_1(\ell)$.
\end{notation}
\begin{definition}
We call \textit{right move} the single action operated on an array by shifting a unit to a $p$ in its right adjacent box on the horizontal axis, whenever that is possible. We call \textit{left move} the reverse move whenever it is possible.
\end{definition}

In the following proposition we set $p=2$.

\begin{proposition}\label{prop:construction} This result and its proof are both due to Joel Louwsma.\\
Let $\alpha_2$ be a given positive integer and $q$ be a given odd prime.\\ If $n\geq \frac{q-\qmat}{2^{\alpha_2}}+N_1(\qmat)+\alpha_2$, then there exists a solution to $\sum_{i=1}^n \frac{1}{x_i}=1$ in $2^a q^b$, where $b\geq 0$ and $0\leq a\leq\al_2$ with $2$ and $q$ each appearing at least once. 
\end{proposition}

\begin{proof} Due to Joel Louwsma.
Write $\qmat =\overline{k_2k_3\dotsm k_{\al_2}1}^{(2)}$
and observe that
\[
\begin{array}{cccccc}
\frac{q-\qmat}{2^{\alpha_2}}&k_2&k_3&\cdots&k_{\alpha_2}& 1  \\
0& 1&1&\cdots&1&1
\end{array}
\]
is a solution in $m$ variables with $m=\frac{q-\qmat}{2^{\alpha_2}}+N_1(\qmat)+\alpha_2$. We will show how to build a solution in $n> m$ variables from this solution in $m$ variables. This will be done by operating some right moves on the first row (with each such move resulting in adding a single variable) or by adding some intermediate rows of the form
\[
\begin{array}{cccccc}\frac{q-\qmat}{2^{\al_2}}&k_2&k_3&\cdots&k_{\al_2}&0\end{array}
\]
between the first row and the last row, or by doing both.
Explicitly, divide $n-m$ by $\frac{q-\qmat}{2^{\al_2}}+N_1(\qmat)-1$ (note, we cannot have both $q=\qmat$ and $N_1(\qmat)=1$. This ensures that the dividend is nonzero). Write
\[
n-m=\ell\Big(\frac{q-\qmat}{2^{\al_2}}+N_1(\qmat)-1\Big)+r,\]
with
\[
0\leq r<\frac{q-\qmat}{2^{\al_2}}+N_1(\qmat)-1.
\]
Then, it will suffice to add $\ell$ intermediate rows and perform $r$ right moves on the first row in order to get a solution in $n$ variables.
\end{proof}

The next proposition states some converse of Proposition~\ref{prop:construction} in the general case, under some restrictions. First, we define the notions of ``bottom row" and ``last row" of a solution, and we introduce another useful move, a vertical one this time. 

\begin{definition}
In a solution, we call ``bottom row" the row containing only $p$-powers; we call ``last row" the row corresponding to the minimal $q$-valuation appearing in the solution when this minimal $q$-valuation is not equal to zero. 
\end{definition}
For instance, the solution provided in $\S\,1$ is a solution with a bottom row; we give below an instance of a solution which has a last row instead. Take $n=7$, $p=3$, $q=5$ and $\al_3=2$. Our program returns a total of $22$ solutions. Amongst those, only one of them has a last row, namely:
\[
\begin{ytableau}
4 & 3 & 0 
\end{ytableau}\;,\;
\text{with $\al_5=1$.}
\]
Note, in that special case, the last row is also the top row. Note also that there exist values of $n$, $p$, $q$ and $\al_p$ for which there is not a single solution with a last row. This is for instance the case when $n=13$, $p=2$, $q=91$ and $\al_2=4$. Namely, there are exactly $6$ solutions in that case, which are obtained from running our program on these values, and all of these solutions have bottom rows. 
\begin{definition} Whenever a box of the array contains an integer which is at least $q$, we call \textit{push down move} the single action performed on the array by shifting the $q$ units down along the vertical axis to a unit in the box below, while leaving all the other integers unchanged.
A push down move may never be performed on the bottom row of the array.
\end{definition}


\begin{proposition}\label{prop:converse} Fix the prime $p$ and the highest $p$-valuation in a solution $\al_p\geq 1$. \\
Let $k$ be an integer with $1\leq k\leq \al_p$ and assume that $q\geq p^{\alpha_p}\Big((p-1)(2\alpha_p-k)+2\Big)-2$. Then, every solution to $\sum_{i=1}^n \frac{1}{x_i}$ in $p^a q^b$ with $b\geq 0$, $0\leq a\leq\al_p$ and $p$ and $q$ each appearing at least once, must have a number $n$ of integer variables satisfying to
\[n\geq \text{Min}_{1\leq s\leq p^k-1}\Big[\frac{sq-\lb sq\rb_{\al_p}}{p^{\al_p}}+N\big(\lb qs\rb_{\al_p}\big)\Big]+(\al_p-k)(p-1)+1.\]
\end{proposition}

\begin{proof}
Let $k$ be an integer with $1\leq k\leq \al_p$. The assumption on $q$ is equivalent to
\begin{equation}\label{eq:inone}
\frac{(p^k-1)q-1}{p^{\al_p}}+(p-1)\al_p+(\al_p-k)(p-1)+1\leq \frac{p^kq-p^{\al_p}+1}{p^{\al_p}}.
\end{equation}
Suppose we have a solution. Up to decreasing the number of variables by operating some left moves on the top row, we may assume that the latter row is of the form
\[
\begin{array}{cccc} k_1&a_2&\dots&a_{\al_p+1},\end{array}
\]
with $a_i\in\lb 0,\dots,p-1\rb,$ $\forall i=2,\dots,\al_p+1$.
Moreover, the fact that the row is admissible forces:
\[
k_1=\frac{qs-\lb qs\rb_{\al_p}}{p^{\al_p}},\;\;\text{some integer $s\geq 1$.}
\]
\indent Notice when $s\geq p^k$ that:
\[
k_1\geq \frac{p^kq-p^{\al_p}+1}{p^{\al_p}}.
\]
Obviously,
\begin{equation}\label{eq:intwo}\text{Min}_{1\leq s\leq p^k-1}\Big[\frac{sq-\lb sq\rb_{\al_p}}{p^{\al_p}}+N\big(\lb qs\rb_{\al_p}\big)\Big]\leq \frac{(p^k-1)q-1}{p^{\al_p}}+(p-1)\al_p.\end{equation}
By a combination of inequalities \eqref{eq:inone} and \eqref{eq:intwo}, we then obtain the result in the case when $s\geq p^k$. 

When $1\leq s\leq p^k-1$ and there are exactly two rows representing the solution, namely the top row and the bottom row, then the bound is obtained without any assumption on $q$. Indeed, reducing the first row by a push right push down operations yields the row
\begin{equation}\label{eq:rownumber}
\begin{array}{ccccc}0&0&\dots&0&s,\end{array}
\end{equation}
where the highest power of $q$ has now been decreased by one.
By forthcoming Lemma~\ref{lem:uniqueness} of $\S\,3$, there is a unique configuration of variables which makes Row~\eqref{eq:rownumber} admissible as a bottom row. Moreover, it requires at least $(\al_p-k)(p-1)+1$ integer variables.

However, if this is not the bottom row of a solution, we want to make this row admissible and reduce it by push right and push down moves. We proceed so repeatedly, until one of the following two situations arises. Either after admissibility, the push value is one of $\lb 1,\dots,p^k-1\rb$ and after push down, the latter value lies in the bottom row, or the push value is $t\geq p^k$. In the first situation, the reasoning from before applies and we thus obtain the inequality of Proposition~\ref{prop:converse}; in the second situation, we started from a row~\eqref{eq:rownumber}
with $1\leq s\leq p^k-1$ and added some variables in order to make this row admissible with resulting push value $t\geq p^k$. Up to decreasing the number of variables by using some left moves, the admissible row reads
\[
\begin{array}{ccccc}k^{'}_1&a^{'}_2&\dots&a^{'}_{\al_p}&(s+\kappa),\end{array}
\]
with $0\leq a^{'}_i,\kappa\leq p-1,\;\forall 2\leq i\leq \al_p$ and
\[k^{'}_1=  \frac{tq-s-\lb tq-s\rb_{\al_p}}{p^{\al_p}}.\]

We get: \begin{equation}\label{eq:inthree}n\geq \frac{p^kq-p^{\al_p}+1}{p^{\al_p}}+\frac{\kappa\,p^{\al_p}-s}{p^{\al_p}}.\end{equation}
We now distinguish between several cases:

$(i)$ If $\kappa\geq 1$, then Inequality~\eqref{eq:inthree} imlies a fortiori that 
\begin{equation}\label{eq:infour}n\geq  \frac{p^kq-p^{\al_p}+1}{p^{\al_p}},\end{equation}
and we conclude like before, using Inequalities \eqref{eq:inone} and \eqref{eq:intwo}. 

$(ii)$ If $\kappa =0$ and at least one of the $a_i$'s with $2\leq i\leq\al_p$ is nonzero, then Inequality~\ref{eq:infour} still holds, thus the result. 

$(iii)$ If $\kappa=0$ and for all $i$ with $2\leq i\leq \al_p$, $a_i=0$, then we have $\lb tq-s\rb_{\al_p}=0$, so that 
\[n\geq \frac{p^kq-p^k+1}{p^{\al_p}}\geq \frac{p^kq-p^{\al_p}+1}{p^{\al_p}}.\] Again, the latter inequality suffices in order to conclude. 
\end{proof}
The converse of Proposition~\ref{prop:converse} does not hold. Indeed, it suffices to consider the case when $p=3$, $\al_3=2$ and $q=53$. The prime $q$ satisfies to the bound of Proposition~\ref{prop:converse}, where we set $k:=2$, since $53\geq 52$. Moreover, we have $53\pmod 9=8=2+2.3$, and so evaluating the bracket when $s=1$ yields the integer $9$. If the converse of Proposition~\ref{prop:converse} were true, it would suffice to have $n\geq 10$ in order to have a solution with decomposition length $n$. However, our program returns no solution when $n=10$, $p=3$, $q=53$ and $\al_3=2$. 

The case of application of Proposition~\ref{prop:converse} with $p=2$ and $k=1$ is special in that, joint with Proposition~\ref{prop:construction}, it provides a necessary and sufficient condition for the existence of a solution, under some restrictions. The result appears below. 
\begin{corollary}\label{cor:cns} Let $\alpha_2$ be a given positive integer and let $q$ be a given odd prime satisfying to $q\geq 2^{\al_2}(2\al_2+1)-2$. 
Then, there exists a solution to $\sum_{i=1}^n \frac{1}{x_i}=1$ in $2^a q^b$, where $b\geq 0$ and $0\leq a\leq\al_2$ with $2$ and $q$ each appearing at least once, if and only if $n\geq \frac{q-\qmat}{2^{\alpha_2}}+N_1(\qmat)+\alpha_2$.
\end{corollary}
The condition on $q$ and $\al_2$ cannot be waived to give a necessary and sufficient condition on the number of variables for the existence of a solution in the case when $p=2$. Namely, take $q=3$ and $\al_2=5$. Then, the following array (with bottom row)
\[\begin{ytableau}1&0&0&0&0&1\\
0&1&0&1&0&1\end{ytableau}\]
provides a solution with $5$ variables and $5<7$. \\

The corollary below provides an upper bound for the primes q involved in solutions when the prime $p$, the number $n$ of unknown variables and the table width $\al_p+1$ are set.
\begin{corollary}\label{cor:qmax}
Fix the number $n$ of unknowns and the width $\al_p+1$ of the table. The odd primes $q$ involved in the solutions to the Diophantine equation in $n$ integer variables of the form $p^aq^b$ with $b\geq 0$, $0\leq a\leq\al_p$ and $p$ and $q$ each appearing at least once, must satisfy
\[
q\leq Max\Big(p^{\al_p}\big((p-1)(2\al_p-k)+2\big)-1,p^{\al_p}\big(n-1-(p-1)(\al_p-k)\big)-1\Big),\]
for each integer $k$ with $1\leq k\leq\al_p$. 
\end{corollary}

\begin{proof} Let $k$ be an integer with $1\leq k\leq\al_p$.  If $q>p^{\al_p}\big((p-1)(2\al_p-k)+2\big)-1$, then $q$ satisfies in particular to the bound of Proposition~\ref{prop:converse}. The latter proposition thus applies and yields:
\[s_0q\leq p^{\al_p}\Big(n-N(\lb qs_0\rb_{\al_p})-(\al_p-k)(p-1)-1\Big)+\lb s_0q\rb_{\al_p},\]
where we have assumed the minimum is attained at $s=s_0$ with $1\leq s_0\leq p^k-1$.
Then, we have: \[q\leq p^{\al_p}(n-1-(\al_p-k)(p-1)-1)+p^{\al_p}-1=p^{\al_p}\big(n-(p-1)(\al_p-k)-1\big)-1.\]
\end{proof}
We deduce the following improvements on the bound $p^{\al_p}n$ in the cases when $n\geq (p-1)\al_p+3$.

\begin{corollary}\label{cor:boundimprovement}
Fix the number $n$ of unknowns and the width $\al_p+1$ of the table. The odd primes $q$ involved in the solutions to the Diophantine equation in $n$ integer variables of the form $p^aq^b$ with $b\geq 0$, $0\leq a\leq\al_p$ and $p$ and $q$ each appearing at least once, must satisfy
\[q< \begin{cases}p^{\al_p}n&\nts\!\!\text{if $n-3<(p-1)\al_p$}\\
p^{\al_p}(n-1)&\nts\!\!\text{if $(p-1)\al_p\leq n-3\leq(p-1)(\al_p+1)$}\\
p^{\al_p}\big((p-1)(\al_p+1)+2\big)&\nts\!\!\text{if $(p-1)(\al_p+1)\leq n-3\leq(p-1)(\al_p+2)$}\\
p^{\al_p}(n-p)&\nts\!\!\text{if $(p-1)(\al_p+2)\leq n-3\leq(p-1)(\al_p+3)$}\\
p^{\al_p}\big((p-1)(\al_p+2)+2\big)&\nts\!\!\text{if $(p-1)(\al_p+3)\leq n-3\leq(p-1)(\al_p+4)$}\\
\qquad\qquad\vdots&\qquad\qquad\qquad\qquad\;\;\;\;\;\vdots\\
p^{\al_p}(n-1-(p-1)(\al_p-3)\big)&\nts\!\!\text{if $(p-1)(3\al_p-6)\leq n-3\leq(p-1)(3\al_p-5)$}\\
p^{\al_p}((p-1)(2\al_p-2)+2)&\nts\!\!\text{if $(p-1)(3\al_p-5)\leq n-3\leq(p-1)(3\al_p-4)$}\\
p^{\al_p}(n-1-(p-1)(\al_p-2)\big)&\nts\!\!\text{if $(p-1)(3\al_p-4)\leq n-3\leq(p-1)(3\al_p-3)$}\\
p^{\al_p}((p-1)(2\al_p-1)+2)&\nts\!\!\text{if $(p-1)(3\al_p-3)\leq n-3\leq(p-1)(3\al_p-2)$}\\
p^{\al_p}\big(n-1-(p-1)(\al_p-1)\big)&\nts\!\!\text{if $n-3\geq(p-1)(3\al_p-2)$}\end{cases}\]
\end{corollary}
\begin{proof} Corollary~\ref{cor:boundimprovement} is a straightforward consequence of the bounds of Corollary~\ref{cor:qmax}. Indeeed, let $k$ be an integer with $1\leq k\leq\al_p$. We have: 
\[(p-1)(2\al_p-k)+2\leq n-1-(p-1)(\al_p-k)\Leftrightarrow n-3\geq (p-1)(3\al_p-2k).\] The last inequality of Corollary~\ref{cor:boundimprovement} then follows from applying Corollary~\ref{cor:qmax} with $k=1$. When $k=\al_p$, and $n-3<(p-1)\al_p$, Corollary~\ref{cor:qmax} implies $q<p^{\al_p}\big((p-1)\al_p+2\big)$, but $(p-1)\al_p+2>n-1$, hence there is no improvement with respect to the original bound of $\S\,1$. Let $k$ be an integer with $1\leq k\leq\al_p-1$. Suppose now that 
\[(p-1)(3\al_p-2k-2)\leq n-3\leq (p-1)(3\al_p-2k).\]
Then, Corollary~\ref{cor:qmax} forces:
\[\left\lbrace\begin{array}{l} q<p^{\al_p}\big(n-1-(p-1)(\al_p-k-1)\big),\\\\
q<p^{\al_p}\big((p-1)(2\al_p-k)+2\big).
\end{array}\right.\]
Moreover, we have:\[n-1-(p-1)(\al_p-k-1)\geq (p-1)(2\al_p-k)+2\Leftrightarrow n-3\geq (p-1)(3\al_p-2k-1).\]
It follows that:

\begin{equation}\label{eq:firstcase}\begin{split}
\text{If } (p-1)(3\al_p-2k-2)\leq n-3&\leq (p-1)(3\al_p-2k-1),\\&\text{then }q<p^{\al_p}\big(n-1-(p-1)(\al_p-k-1)\big).
\end{split}\end{equation}
\begin{equation}\label{eq:secondcase}\begin{split}
\text{If } (p-1)(3\al_p-2k-1)\leq n-3&\leq (p-1)(3\al_p-2k),\\&\nts\text{then }q<p^{\al_p}\big((p-1)(2\al_p-k)+2\big).
\end{split}\end{equation}
By applying Inequalities~\ref{eq:firstcase} and~\ref{eq:secondcase} for $k$ varying between $1$ and $\al_p-1$, we obtain all the inequalities of Corollary~\ref{cor:boundimprovement} corresponding to the intermediate rows located in between the first row and the last row. 
\end{proof}
The next paragraph uses the previously stated results in order to discuss the repartition of the primes $q$ occurring in solutions in the case when $p=2$.

Fix $n$ the number of unknowns and fix an odd integer $c<2^{\alpha_2}$. Proposition~\ref{prop:construction} shows that there is always a solution for sufficiently small values of $q$ of the form $c+2^{\al_2}l$. Specifically, the construction offered in the proof provides a solution whenever
\[
q\leq 2^{\al_2}(n-N_1(c)-\alpha_2)+c.
\]

We discuss below a few base values for $\al_2$.
\begin{enumerate}[label=\textup{\roman*.},ref=\textup{\roman*}]
\item If $\al_2=1$, the sole residue for $q$ is $c=1$. Then there exists a solution to the equation for every odd prime $q\leq 2n-3$.
\item If $\al_2=2$, as any odd prime is congruent to $1$ or $3$ modulo $4$, the possibilities for $c$ are $c=1$ or $c=3$. Then all the primes congruent to $1$ modulo $4$ satisfying to $q\leq 4n-11$ occur in solutions and all those congruent to $3$ moduo $4$ satisfying to $q\leq 4n-13$ occur in solutions. By gathering both facts, we see that all the odd primes $q\leq 4n-11$ occur in solutions.
\item If $\al_2=3$, we draw a table gathering the results.
\[
\begin{tabular}{|c|c|c|c|c|}
\hline
c & 1 & 3 & 5 & 7 \\ \hline
$q\,\leq$ & 8n-31 & 8n-37 & 8n-35 & 8n-41 \\
\hline\end{tabular}\]
\item If $\al_2=4$, the results get displayed in the table below.\\
\[ \begin{tabular}{|c|c|c|c|c|c|c|c|c|}
\hline
c & 1 & 3 & 5 & 7 & 9 & 11 & 13 & 15\\ \hline
$q\,\leq$ & 16n-79 & 16n-93 & 16n-91 & 16n-105 & 16n-87 & 16n-101 & 16n-99& 16n-113  \\ \hline\end{tabular}\]
\end{enumerate}
\[\begin{array}{l}\end{array}\]
Going back to ii., we state the following result. 
\begin{proposition}\label{prop:aleqtwo}
Let $q$ be an odd prime with $q\neq 3$. Then, there exists a solution to $\sum_{i=1}^n1/x_i=1$ in $2^aq^b$ with the integers $a$ and $b$ satisfying to $b\geq 0$, $0\leq a\leq 2$, and $2$ and $q$ each appearing at least once, if and only if 
$q\leq 4n-11$.
\end{proposition}
\begin{proof}
By the discussion above, the sufficient condition holds. Moreover, by Corollary~\ref{cor:cns}, the necessary condition holds as soon as $q\geq 18$. Further, our program returns no solution in the cases when 
\[\begin{array}{ccc} q=5& \&& n=3,\\
q=7&\&&n\in\lb 3,4\rb,\\
q=11&\&&n\in\lb 3,4,5\rb,\\
q=13&\&&n\in\lb 3,4,5\rb,\\
q=17&\&&n\in\lb 3,4,5,6\rb.\
\end{array}\]

\end{proof}
We note that Proposition~\ref{prop:aleqtwo} does not hold when $q=3$. Indeed,  the array \[\begin{ytableau}1&1&0\\0&1&0\\
\end{ytableau}\] provides a solution in three variables, with bottom row and with highest $2$-valuation less than or equal to $2$. 

Our results have their limitations. For instance, notice from the table of iii. that we cannot statue on $8n-33$ when $n=8$. Fortunately, our program provides an answer to the case when $n=8$, $p=2$, $q=31$ and $\al_2=3$ by returning no solution in that case. However, we are able to decide when $n=13$. Indeed, since $71>54$, Proposition~\ref{prop:converse} applies and shows that prime number $71$ cannot occur in the solutions with highest $2$-valuation $3$. This shows in particular that there exist gaps in the primes $q$ occurring in solutions since both $67$ and $73$ occur when $n=13$, while the prime $71$ lying in between does not.

\section{Description of the algorithm}

At each step $i$ of the algorithm, we keep track of a number $l_i$ of left moves so that a new top row or a bottom row has a simplified admissible shape.
At the end, we must perform on each row $R_i$ a number $l_i$ of right moves. We must do so in all possible ways.
Let $N_i$ be the number of ways to expand row $R_i$ by operating $l_i$ right moves. Then for each reduced solution containing a bottom row, which has $\al_q+1$ rows (here $\al_q$ denotes the highest $q$-valuation occurring in the reduced solution, hence in the solution) produced by the algorithm, we thus obtain $\prod_{i=1}^{\al_q+1} N_i$ solutions; for each reduced solution rather having a last row and for which $r$ reduced rows have been made admissible, we obtain $\prod_{i=1}^{r} N_i$ solutions. Obviously, we obtain a set of solutions whose elements are not necessarily two by two distinct. Thus, we must make sure to list and count each solution only once. \\

\textbf{Initialize.} At first, we record $l_1$ left moves and a resulting reduced admissible top row of shape
\[\begin{array}{cc}\frac{sq-\lb sq\rb_{\al_p}}{p^{\al_p}}&\text{$p$-adic expansion of $\lb sq\rb_{\al_p}$,}\end{array}\]
with the expansion above written from right to left and possibly containing $0$'s on both ends in order to fill out all of the $\al_p$ columns (with the same convention applying throughout the rest of the paper), and with $s$ a positive integer which we record under the form of a couple $(s_1,\tilde{s_1})$ so that:
\[s=p^{\al_p}\tilde{s_1}+s_1,\;\; 0\leq s_1<p^{\al_p}\]
Note that $s_1$ and $\tilde{s_1}$ cannot be both zero, as otherwise the full row would be zero and cannot be a top row.
Since we have:
\[\frac{sq-\lb sq\rb_{\al_p}}{p^{\al_p}}=\tilde{s_1}q+\frac{s_1q-\lb s_1q\rb_{\al_p}}{p^{\al_p}},\]
the top row rewrites as:
\[\begin{array}{cc}\tilde{s_1}q+\frac{s_1q-\lb s_1q\rb_{\al_p}}{p^{\al_p}}&\text{$p$-adic expansion of $\lb s_1q\rb_{\al_p}$}\end{array}\]
\indent If $s_1=0$, $l_1\neq 0$ and we have already used all our integer variables by making the top row admissible while taking also into account the left moves, then we record a solution only if $\tilde{s_1}=q^b$ for some $b\geq 0$. The latter solution then has highest $q$-valuation $b+1$. The condition on $l_1$ is to ensure that $p$ occurs at least once as a factor in the denominators of a solution. 

If $\tilde{s_1}=0$, the push right, push down moves result in the following row:
\[\begin{array}{cccc} 0&\dots &0&s_1\end{array}\]
We can make this row admissible either as a bottom row or as a new top row.
First and foremost, we deal with the first situation which is specific to the fact that $\tilde{s_1}=0$. Again, we record a number $l_2$ of left moves, and so we must complete the row with integers $k_1$, $k_2$, $\dots$, $k_{\al_p+1}$ so that for each $i$ with $2\leq i\leq \al_p+1$, we have $0\leq k_i\leq p-1$ and the row
\[\begin{array}{ccccc}k_1&k_2&\dots&k_{\al_p}&(k_{\al_p+1}+s_1)\end{array}\] is admissible as a bottom row.
\begin{lemma}\label{lem:uniqueness}
There is a unique configuration of the $k_i$'s which makes the row admissible as a bottom row.
\end{lemma}
\begin{proof}
Recall that $s_1<p^{\al_p}$. Then, $s_1$ can be decomposed in base $p$ over the last $\al_p$ columns. In this decomposition, the first column containing a non-zero pit is the $(v_p(s_1)+1)$-th column from the right, where $v_p(s_1)$ denotes the $p$-valuation of $s_1$. And so,
\[s_1=\overline{a_{\al_p-1}\dots a_{v_p(s_1)}\dots a_0}^{(p)}\qquad\text{with $\forall 0\leq k<v_p(s_1),\,a_k=0.$}\]
By a slight abuse of notation and in order to ease the discussion right below, the base $p$ decomposition of $s_1$ above may contain some $0$'s to the left hand side. 
We claim that for each $j$ with $\al_p+2-v_p(s_1)\leq j\leq \al_p+1$, we have $k_j=0$ when $v_p(s_1)\neq 0$. As otherwise, the highest $p$-valuation would occur a number of times which is not divisible by $p$. This is impossible. By the same argument, we may claim that:
\[k_{\al_p+1-v_p(s_1)}=p-a_{v_p(s_1)},\]
including in the case when $v_p(s_1)=0$. 
Then, the box in the $(v_p(s_1)+1)$-th column from the right contains a $p$.
We see by operating a left move that it forces inductively:
\[\forall 2\leq j\leq \al_p-v_p(s_1),\,k_j=p-1-a_{\al_p+1-j}.\]
At the end of the inductive process, there is a $p$ in the box of the second column, which shifted by left move translates into a $1$ in position $(1,1)$. This forces at last $k_1=0$ in order to obtain a solution.
\end{proof}
In light of Lemma \ref{lem:uniqueness} , we introduce a new definition.
\begin{definition}\label{def:not}
When $s=\overline{s_{\al_p(s)}\dots s_{v_p(s)}0\dots 0}^{(p)}$, define
\[\text{NOT}^{(p)}(s)=\overline{s^{'}_{\al_p(s)}\dots s^{'}_{v_p(s)+1}s^{'}_{v_p(s)}0\dots 0}^{(p)},\]
with \begin{equation*}
s^{'}_j=\begin{cases} p-1-s_j&\text{when $v_p(s)+1\leq j\leq\al_p(s)$}\\p-s_j&\text{when $j=v_p(s)$}
\end{cases}
\end{equation*}
\end{definition}
In the case when $p=2$, the operator on $s$ applies NOT on the last $\al_2(s)-v_2(s)$ bits of $s$, while leaving all the other bits unchanged.

Suppose now $\tilde{s_1}$ is either $0$ or positive. In both cases, we can make the row issued from the push right, push down operations an admissible top row. Like previously, we record a number $l_2$ of left moves to be applied as right moves in all possible ways at the end.
The row to be completed reads
\[\begin{array}{ccccc} \tilde{s_1}&0&\dots&0&s_1,\end{array}\]
with all the $k_i$'s to be added belonging to $\lb 0,\dots,p-1\rb$, except possibly $k_1$.
The admissibility condition from Lemma $1$ imposes:
\[k_1=\frac{qs_2-s_1-\lb qs_2-s_1\rb_{\al_p}}{p^{\al_p}}+q\tilde{s_2}-\tilde{s_1},\]
for some nonnegative integers $\tilde{s_2}$ and $s_2$ such that $0\leq s_2<p^{\al_p}$ and $s_2$, $\tilde{s_2}$ not both zero. The other $k_i$'s get added in such a way that: \[\lb qs_2-s_1\rb_{\al_p}=\overline{k_2\dots k_{\al_p+1}}^{(p)}.\]

\textbf{Iteration.} Suppose Row $(i-1)$ with $i\geq 3$ has been filled so that it is admissible as a top row. Reducing this row leads to the following row, one $q$-valuation down:
\[\begin{array}{ccccc} \tilde{s_{i-1}}&0&\dots&0&s_{i-1}.\end{array}\]

If the number of variables left is zero, by taking also into account the left moves, and if $s_{i-1}=0$, check whether $\tilde{s_{i-1}}$ is a power of $q$, namely $q^b$ for some $b\geq 0$. If so, then record a solution with highest $q$-valuation equal to $b+i-1$. Also keep track of the integer $b+2$ indicating the height of the solution with last row. 

Otherwise, record a number $l_i$ of left moves and the row solution $i$ defined by:
\[\begin{array}{ccc}R_i: & q\tilde{s_i}-\tilde{s_{i-1}}+\frac{qs_i-s_{i-1}-\lb qs_i-s_{i-1}\rb_{\al_p}}{p^{\al_p}}&\text{$p$-adic expansion of $\lb qs_i-s_{i-1}\rb_{\al_p}$}\end{array},\]
for some adequate nonnegative integers $\tilde{s_i}$ and $s_i$ with $0\leq s_i<p^{\al_p}$ and $s_i$ and $\tilde{s_i}$ not both equal to zero. 

Independently and only when $\tilde{s_{i-1}}=0$, add a number $l_i$ of left moves and check for a solution whose bottom row is
\[\begin{array}{ccccc} 0&\underbrace{\text{NOT$^{(p)}$($p$-adic expansion of $s_{i-1}$)}}&0&\dots&0,\\
&2\leq\text{columns}\leq \al_p-v_p(s_{i-1})+1&&&\end{array}\]
\noindent pending on the number of variables left. The highest $q$-valuation of such a solution equals $i-1$ and the  height of such solution is by convention $1$. \\

\textbf{Output before the right moves.}
The output before right moves for a solution with a bottom row consists of a certain number of rows, filled in the way we provide below, and a nonnegative integer associated with each row, instructing on the number of right moves to be performed on that row in order to get the final solution in the adequate number of integer variables. In the rows displayed below, $\al_q$ denotes the highest $q$-valuation of the solution.

\[\begin{array}{ccc}
\nts\nts l_1,&q\tilde{s_1}+\frac{qs_1-\lb qs_1\rb_{\al_p}}{p^{\al_p}}&\text{$p$-adic expansion of $\lb qs_1\rb_{\al_p}$}\\
\nts \nts l_2,&q\tilde{s_2}-\tilde{s_1}+\frac{qs_2-s_1-\lb qs_2-s_1\rb_{\al_p}}{p^{\al_p}}&\text{$p$-adic expansion of $\lb qs_2-s_1\rb_{\al_p}$}\\
&&\\
&\vdots&\\
&&\\
\nts\nts l_i,&\nts q\tilde{s_i}-\tilde{s_{i-1}}+\frac{qs_i-s_{i-1}-\lb qs_i-s_{1-1}\rb_{\al_p}}{p^{\al_p}}&\nts\nts\nts\text{$p$-adic expansion of $\lb qs_i-s_{i-1}\rb_{\al_p}$}\\
&&\\
&\vdots&\\
&&\\
\nts\nts l_{\al_q},&\nts\nts\nts\nts -\tilde{s_{\al_q-1}}+\frac{qs_{\al_q}-s_{\al_q-1}-\lb qs_{\al_q}-s_{\al_q-1}\rb_{\al_p}}{p^{\al_p}}&\nts\!\!\text{$p$-adic expansion of $\lb qs_{\al_q}-s_{\al_q-1}\rb_{\al_p}$}\\
&&\\
\nts\nts l_{\al_q+1},&0&\!\!\underbrace{\text{NOT}^{(p)}(\text{$p$-adic expansion of $s_{\al_q}$)}}\;\;0\;\;\dots0\\
&&2\leq\text{columns}\leq \al_p-v_p(s_{\al_q})+1\qquad\qquad\end{array}\]
\begin{remark}\label{re:def}
In the two paragraphs ``Iteration" and ``Output before right moves" of the above, we naturally extended Definition~\ref{def:not} on the operator NOT$^{(p)}$ to the case when, by abuse of notations, some $0$'s get added at the end of the base $p$-expansion. 
\end{remark}
Notice that $s_1,\tilde{s_1}$ cannot be both zero forces inductively $s_i,\tilde{s_i}$ cannot be both zero when $i\geq 2$.
Consequently, whenever $i\geq 1$, \\\\
\begin{tabular}{lll}
If&$s_i=0$,&the reduced row down cannot be a bottom row.\\
If&$\tilde{s_i}=0$,&the reduced row down can be either a new top row or a bottom row. \end{tabular}\\\\
We will now bound - in a brute force manner -  the admissible values for $l_i,s_i,\tilde{s_i}$ at each step $i$ of the algorithm. \\
Initially, the number of left moves must be less than $\al_p\lf\frac{n}{p}\rf$.
Thus, it will be sufficient to run the algorithm for
\[\left\lb\begin{array}{l}0\leq l_1<\al_p\lf\frac{n}{p}\rf,\\
0\leq s_1<p^{\al_p},\\
0\leq\tilde{s_1}\leq\lf\frac{n}{q}\rf,\\
(s_1,\tilde{s_1})\neq (0,0).
\end{array}\right.\]
After initialization, the number of remaining variables is:
\[n_1(l_1,s_1,\tilde{s_1})=n-l_1(p-1)-q\tilde{s_1}-\frac{qs_1-\lb qs_1\rb_{\al_p}}{p^{\al_p}}-N\big(\lb qs_1\rb_{\al_p}\big).\]
If $n_1(l_1,s_1,\tilde{s_1})\leq 0$, then proceed to the next values for $l_1,s_1,\tilde{s_1}$. Otherwise, store the data from the first row. \\
At each step $i\geq 2$, whenever $\tilde{s_{i-1}}=0$ and $n_{i-1}(l_{i-1},s_{i-1},0)>0$, try
\[0\leq l_i< \al_p\Big\lf \frac{n_{i-1}(l_{i-1},s_{i-1},0)}{p}\Big\rf.\]
If \[p\big(l_i+\al_p-v_p(s_{i-1})\big)-(l_i+\al_p-v_p(s_{i-1})-1)-N(s_{i-1})=n_{i-1}(l_{i-1},s_{i-1},0),\]
then store a solution with bottom row. Anyhow, try:
\[\left\lb\begin{array}{l}0\leq l_i< \al_p\lf\frac{n_{i-1}(l_{i-1},s_{i-1},\tilde{s_{i-1}})}{p}\rf,\\
0\leq s_i<p^{\al_p},\\
n_{i-1}(l_{i-1},s_{i-1},\tilde{s_{i-1}})+\tilde{s_{i-1}}-\frac{qs_i-s_{i-1}-\lb qs_i-s_{1-1}\rb_{\al_p}}{p^{\al_p}}\geq 0,\\\\
0\leq\tilde{s_i}\leq\lf\frac{n_{i-1}(l_{i-1},s_{i-1},\tilde{s_{i-1}})+\tilde{s_{i-1}}-(qs_i-s_{i-1}-\lb qs_i-s_{1-1}\rb_{\al_p})/p^{\al_p}}{q}\rf,\\
(s_i,\tilde{s_i})\neq (0,0),\\
q\tilde{s_i}-\tilde{s_{i-1}}+\frac{qs_i-s_{i-1}-\lb qs_i-s_{1-1}\rb_{\al_p}}{p^{\al_p}} \geq 0.
\end{array}\right.\]
Let \begin{multline*} n_i(l_i,s_i,\tilde{s_i}):=n_{i-1}(l_{i-1},s_{i-1},\tilde{s_{i-1}})-l_i(p-1)\\+\tilde{s_{i-1}}-q\tilde{s_i}-\frac{qs_i-s_{i-1}-\lb qs_i-s_{i-1}\rb_{\al_p}}{p^{\al_p}}-N(\lb qs_i-s_{i-1}\rb_{\al_p}).\end{multline*}
\begin{itemize}
\item If $n_i(l_i,s_i,\tilde{s_i})<0$, then proceed to the next try. 
\item If $n_i(l_i,s_i,\tilde{s_i})=0$, then check for a solution with last row. That is check whether 
\[\left\lb\begin{array}{l}
n_{i-1}+\tilde{s_{i-1}}+(s_{i-1}+\lb -s_{i-1}\rb_{\al_p})/p^{\al_p}\geq q,\\\\
\text{$s_i=0$ and $\tilde{s_i}=q^{b}$, for some integer $b$ with:}\\
0\leq b\leq log_q\Big\lf \frac{n_{i-1}+\tilde{s_{i-1}}+(s_{i-1}+\lb -s_{i-1}\rb_{\al_p})/p^{\al_p}}{q}\Big\rf. 
\end{array}\right.\]
If so, store a solution with last row, whose highest $q$-valuation is $b+i$ and whose height is $b+2$. 
\item If $n_i(l_i,s_i,\tilde{s_i})>0$, then store a new intermediate solution row. \\
\end{itemize}

\textbf{Output after right moves.}
After performing the right moves on the reduced solutions, discard those unicolumn solutions with last row for which the prime $p$ does not appear as a factor in the denominators (for those solutions, at each step the algorithm chooses a null number of left moves). For instance, when $n=q$, $p\in\mathcal{P}-\lb q\rb$ and $\al_p\geq 1$, the following solution of height $2$ must be discarded.
\[\begin{ytableau} q&0&\dots&0\end{ytableau}\] 


\end{document}